\documentclass[11pt,reqno]{amsart}
\usepackage{amsthm}
\usepackage{amssymb}
\usepackage{latexsym}
\usepackage{tikz}
\usepackage{multicol}
\usepackage{verbatim,enumerate}
\usepackage{amscd}
\usetikzlibrary{arrows}
\usepackage{float}

\usepackage{hyperref}
\usepackage{amsmath, amscd}

\advance\textwidth by 1.2in \advance\oddsidemargin by -.6in \advance\evensidemargin by -.6in
\usepackage{tikz}
\usetikzlibrary{decorations.markings}
\tikzstyle{vertex}=[ circle, draw, inner sep=0pt, minimum size=2pt,]

\newtheorem{proplab}{Proposition}

\theoremstyle{definition}

\theoremstyle{definition}
\newtheorem{thm}{Theorem}
\newtheorem*{thm*}{Theorem}

\newtheorem{lemlab}{Lemma}

\newcounter{cnt}
 \makeatletter
\def\mydggeometry{\makeatletter\dg@YGRID=1\dg@XGRID=20\unitlength=0.003pt\makeatother}
\makeatother \theoremstyle{remark}

\numberwithin{equation}{section}

\let\bwdg\bigwedge
\def\bigwedge{{\textstyle\bwdg}}

\begin{document}

\newcommand{\thmref}[1]{Theorem~\ref{#1}}
\newcommand{\secref}[1]{Section~\ref{#1}}
\newcommand{\lemref}[1]{Lemma~\ref{#1}}
\newcommand{\propref}[1]{Proposition~\ref{#1}}
\newcommand{\corref}[1]{Corollary~\ref{#1}}
\newcommand{\remref}[1]{Remark~\ref{#1}}
\newcommand{\defref}[1]{Definition~\ref{#1}}
\newcommand{\er}[1]{(\ref{#1})}
\newcommand{\id}{\operatorname{id}}
\newcommand{\ord}{\operatorname{\emph{ord}}}
\newcommand{\sgn}{\operatorname{sgn}}
\newcommand{\wt}{\operatorname{wt}}
\newcommand{\tensor}{\otimes}
\newcommand{\from}{\leftarrow}
\newcommand{\nc}{\newcommand}
\newcommand{\rnc}{\renewcommand}
\newcommand{\dist}{\operatorname{dist}}
\newcommand{\qbinom}[2]{\genfrac[]{0pt}0{#1}{#2}}
\nc{\cal}{\mathcal} \nc{\goth}{\mathfrak} \rnc{\bold}{\mathbf}
\renewcommand{\frak}{\mathfrak}
\newcommand{\supp}{\operatorname{supp}}
\newcommand{\Irr}{\operatorname{Irr}}
\newcommand{\psym}{\mathcal{P}^+_{K,n}}
\newcommand{\psyml}{\mathcal{P}^+_{K,\lambda}}
\newcommand{\psymt}{\mathcal{P}^+_{2,\lambda}}
\renewcommand{\Bbb}{\mathbb}
\nc\bomega{{\mbox{\boldmath $\omega$}}} \nc\bpsi{{\mbox{\boldmath $\Psi$}}}
 \nc\balpha{{\mbox{\boldmath $\alpha$}}}
 \nc\bpi{{\mbox{\boldmath $\pi$}}}
  \nc\bxi{{\mbox{\boldmath $\xi$}}}
\nc\bmu{{\mbox{\boldmath $\mu$}}} \nc\bcN{{\mbox{\boldmath $\cal{N}$}}} \nc\bcm{{\mbox{\boldmath $\cal{M}$}}} \nc\blambda{{\mbox{\boldmath
$\lambda$}}}\nc\bnu{{\mbox{\boldmath $\nu$}}}

\newcommand{\Tmn}{\bold{T}_{\lambda^1, \lambda^2}^{\nu}}

\newcommand{\lie}[1]{\mathfrak{#1}}
\newcommand{\ol}[1]{\overline{#1}}
\makeatletter
\def\section{\def\@secnumfont{\mdseries}\@startsection{section}{1}%
  \z@{.7\linespacing\@plus\linespacing}{.5\linespacing}%
  {\normalfont\scshape\centering}}
\def\subsection{\def\@secnumfont{\bfseries}\@startsection{subsection}{2}%
  {\parindent}{.5\linespacing\@plus.7\linespacing}{-.5em}%
  {\normalfont\bfseries}}
\makeatother
\def\subl#1{\subsection{}\label{#1}}
 \nc{\Hom}{\operatorname{Hom}}
  \nc{\mode}{\operatorname{mod}}
\nc{\End}{\operatorname{End}} \nc{\wh}[1]{\widehat{#1}} \nc{\Ext}{\operatorname{Ext}} \nc{\ch}{\text{ch}} \nc{\ev}{\operatorname{ev}}
\nc{\Ob}{\operatorname{Ob}} \nc{\soc}{\operatorname{soc}} \nc{\rad}{\operatorname{rad}} \nc{\head}{\operatorname{head}}
\def\Im{\operatorname{Im}}
\def\gr{\operatorname{gr}}
\def\mult{\operatorname{mult}}
\def\Max{\operatorname{Max}}
\def\ann{\operatorname{Ann}}
\def\sym{\operatorname{sym}}
\def\loc{\operatorname{loc}}
\def\Res{\operatorname{\br^\lambda_A}}
\def\und{\underline}
\def\Lietg{$A_k(\lie{g})(\bsigma,r)$}
\def\res{\operatorname{res}}

 \nc{\Cal}{\cal} \nc{\Xp}[1]{X^+(#1)} \nc{\Xm}[1]{X^-(#1)}
\nc{\on}{\operatorname} \nc{\Z}{{\bold Z}} \nc{\J}{{\cal J}} \nc{\C}{{\bold C}} \nc{\Q}{{\bold Q}}
\renewcommand{\P}{{\cal P}}
\nc{\N}{{\Bbb N}} \nc\boa{\bold a} \nc\bob{\bold b} \nc\boc{\bold c} \nc\bod{\bold d} \nc\boe{\bold e} \nc\bof{\bold f} \nc\bog{\bold g}
\nc\boh{\bold h} \nc\boi{\bold i} \nc\boj{\bold j} \nc\bok{\bold k} \nc\bol{\bold l} \nc\bom{\bold m} \nc\bon{\bold n} \nc\boo{\bold o}
\nc\bop{\bold p} \nc\boq{\bold q} \nc\bor{\bold r} \nc\bos{\bold s} \nc\boT{\bold t} \nc\boF{\bold F} \nc\bou{\bold u} \nc\bov{\bold v}
\nc\bow{\bold w} \nc\boz{\bold z} \nc\boy{\bold y} \nc\ba{\bold A} \nc\bb{\bold B} \nc\bc{\bold C} \nc\bd{\bold D} \nc\be{\bold E} \nc\bg{\bold
G} \nc\bh{\bold H} \nc\bi{\bold I} \nc\bj{\bold J} \nc\bk{\bold K} \nc\bl{\bold L} \nc\bm{\bold M} \nc\bn{\bold N} \nc\bo{\bold O} \nc\bp{\bold
P} \nc\bq{\bold Q} \nc\br{\bold R} \nc\bs{\bold S} \nc\bt{\bold T} \nc\bu{\bold U} \nc\bv{\bold V} \nc\bw{\bold W} \nc\bz{\bold Z} \nc\bx{\bold
x} \nc\KR{\bold{KR}} \nc\rk{\bold{rk}} \nc\het{\text{ht }}

\nc\toa{\tilde a} \nc\tob{\tilde b} \nc\toc{\tilde c} \nc\tod{\tilde d} \nc\toe{\tilde e} \nc\tof{\tilde f} \nc\tog{\tilde g} \nc\toh{\tilde h}
\nc\toi{\tilde i} \nc\toj{\tilde j} \nc\tok{\tilde k} \nc\tol{\tilde l} \nc\tom{\tilde m} \nc\ton{\tilde n} \nc\too{\tilde o} \nc\toq{\tilde q}
\nc\tor{\tilde r} \nc\tos{\tilde s} \nc\toT{\tilde t} \nc\tou{\tilde u} \nc\tov{\tilde v} \nc\tow{\tilde w} \nc\toz{\tilde z} \nc\woi{w_{\omega_i}}
\nc\chara{\operatorname{Char}}
\title{Conjecture $\mathcal{O}$ holds for some Horospherical Varieties of Picard Rank 1}
\author[Bones, Fowler, Schneider, Shifler]{Lela Bones, Garrett Fowler, Lisa Schneider and Ryan M. Shifler}\address{\noindent Department of Mathematics, Salisbury University, MD 21801}
\email{lbones1@gulls.salisbury.edu, gfowler2@gulls.salisbury.edu, lmschneider@salisbury.edu, rmshifler@salisbury.edu}
\begin{abstract}
Property $\mathcal{O}$ for an arbitrary complex, Fano manifold $X$, is a statement about the eigenvalues of the linear operator obtained from the quantum multiplication of the anticanonical class of $X$. Conjecture $\mathcal{O}$ is a conjecture that Property $\mathcal{O}$ holds for any Fano variety. Pasquier listed the smooth non-homogeneous horospherical varieties of Picard rank 1 into five classes. Conjecture $\mathcal{O}$ has already been shown to hold for the odd symplectic Grassmannians which is one of these classes. We will show that Conjecture $\mathcal{O}$ holds for two more classes and an example in a third class of Pasquier's list. The theory of Perron-Frobenius reduces our proofs to be graph-theoretic in nature.
\end{abstract}
\maketitle
\section{Introduction}
The purpose of this paper is to prove that Conjecture $\mathcal{O}$ holds for some horospherical varieties of Picard rank 1. We recall the precise statement of Conjecture $\mathcal{O}$ for varieties of Picard rank 1, following \cite[Section 3]{GGI}. Let $F$ be a Fano variety, let $K:=K_F$ be the canonical line bundle of $F$, let $F_D$ be a fundamental divisor of $F$, and let $c_1(F):=c_1(-K) \in H^2(F)$ be the anticanonical class. The Fano index of $F$ is $r$, where $r$ is the greatest integer such that $K_F \cong -rF_D$. The small quantum cohomology ring $(QH^*(F), \star)$ is a graded algebra over $\mathbb{Z}[q]$, where $q$ is the quantum parameter. We define the small quantum cohomology in Section~\ref{sec:Section 2.1}. Consider the specialization $H^{\bullet}(F):=QH^*(F)|_{q=1}$ at $q=1$. The quantum multiplication by the first Chern class $c_1(F)$ induces an endomorphism $\hat{c}_1$ of the finite-dimensional vector space $H^{\bullet}(F)$: \[y \in H^{\bullet}(F) \mapsto \hat{c}_1(y):=(c_1(F)\star y)|_{q=1}. \]

Denote by $\delta_0:=\max \{|\delta| : \delta \mbox{ is an eigenvalue of } \hat{c}_1 \}.$ Then Property $\mathcal{O}$ states the following:
\begin{enumerate}
\item The real number $\delta_0$ is an eigenvalue of $\hat{c}_1$ of multiplicity one.
\item If $\delta$ is any eigenvalue of $\hat{c}_1$ with $|\delta|=\delta_0$, then $\delta=\delta_0 \gamma$ for some $r$-th root of unity $\gamma \in \mathbb{C}$, where $r$ is the Fano index of $F$.
\end{enumerate}
The property $\mathcal{O}$ was conjectured to hold for any Fano, complex manifold $F$ in \cite{GGI}. If a Fano, complex, manifold has Property $\mathcal{O}$ then we say that the space satisfies Conjecture $\mathcal{O}$. Conjecture $\mathcal{O}$ underlies Gamma Conjectures I and II of Galkin, Golyshev, and Iritani. The Gamma Conjectures refine earlier conjectures by Dubrovin on Frobenius manifolds and mirror symmetry. Conjecture $\mathcal{O}$ has already been proved for the homogeneous $G/P$ case in \cite{CL}, the odd symplectic Grassmannians in \cite{LMS}, del Pezzo surfaces in \cite{delpezzo}, and projective complete intersections in \cite{projinter}. The Perron-Frobenius theory of nonnegative matrices reduces the proofs that Conjecture $\mathcal{O}$ holds for the homogeneous and the odd symplectic Grassmannian cases to be a graph-theoretic check. This is because Conjecture $\mathcal{O}$ is largely reminiscent of Perron-Frobenius Theory. In this manuscript we will use the same graph-theoretic approach to prove that Conjecture $\mathcal{O}$ holds for some smooth horospherical varieties of Picard rank 1.

Next we recall the definition of a horospherical variety following \cite{GPPS}. Let $G$ be a complex reductive group. A $G$-variety is a reduced scheme of finite type over the field of complex numbers $\mathbb{C}$, equipped with an algebraic action of $G$. Let $B$ be a Borel subgroup of $G$. A $G$-variety $X$ is called spherical if $X$ has a dense $B$-orbit. Let $X$ be a $G$-spherical variety and let $H$ be the stabilizer of a point in the dense $G$-orbit in $X$. The variety $X$ is called {\it horospherical} if $H$ contains a conjugate of the maximal unipotent subgroup of $G$ contained in the Borel subgroup $B$.

Smooth horospherical varieties of Picard rank 1 were classified by Pasquier in \cite{P}. These varieties are either homogeneous or can be constructed in a uniform way via a triple (Type($G$),$\omega_Y$,$\omega_Z$) of representation-theoretic data, where Type($G$) is the semisimple Lie type of the reductive group $G$ and $\omega_Y, \omega_Z$ are fundamental weights. See \cite[Section 1.3]{P} for details. Pasquier classified the possible triples in five classes:
\begin{enumerate}
\item $(B_n, \omega_{n-1},\omega_n)$ with $n \geq 3$;
\item $(B_3, \omega_1,\omega_3)$;
\item $(C_n,\omega_m,\omega_{m-1})$ with $n \geq 2$ and $m \in [2,n]$ (the odd symplectic Grassmannians);
\item $(F_4, \omega_2, \omega_3)$;
\item $(G_2,\omega_1,\omega_2)$.
\end{enumerate}

 In Proposition 3.6 of \cite{Pthesis}, Pasquier showed the triples in the above list are Fano varieties. We are now able to state the main theorem:

\begin{thm}\label{thm:main}
If $F$ belongs to the classes (1) for $n=3$, (2), (3), and (5) of Pasquier's list, then Conjecture $\mathcal{O}$ holds for $F$.
\end{thm}
{\it Acknowledgements:} We would like to thank the anonymous referees for their useful comments and suggestions to improve the presentation of this paper.
\section{Preliminaries}

\subsection{Quantum Cohomology}
\label{sec:Section 2.1}
The small quantum cohomology is defined as follows. Let $(\alpha_i)_i$ be a basis of $H^*(F)$, the classical cohomology ring, and let $(\alpha_i^{\vee})_i$ be the dual basis for the Poincar\'e pairing. The multiplication is given by \[ \alpha_i \star \alpha_j= \sum_{d\geq 0, k} c_{i,j}^{k,d}q^d\alpha_k \] where $c_{i,j}^{k,d}$ are the 3-point, genus 0, Gromov-Witten invariants corresponding to the classes $\alpha_i, \alpha_j$, and $\alpha_k^{\vee}$. We will make use of the quantum Chevalley formula which is the multiplication of a hyperplane class $h$ with another class $\alpha_j$. The result \cite[Theorem 0.0.3]{GPPS} implies that if $F$ belongs to the classes (1) for $n=3$, (2), or (5) of Pasquier's list, then there is an explicit quantum Chevalley formula. The explicit quantum Chevalley formula is the key ingredient used to prove Property $\mathcal{O}$ holds.

\subsection{Sufficient Criterion for Property $\mathcal{O}$ to hold} We recall the notion of the (oriented) quantum Bruhat graph of a Fano variety $F$. The vertices of this graph are the basis elements $\alpha_i \in H^{\bullet}(F):=QH^*(F)|_{q=1}$. There is an oriented edge $\alpha_i \rightarrow \alpha_j$ if the class $\alpha_j$ appears with positive coefficient (where we consider $q>0$) in the quantum Chevalley multiplication $h \star \alpha_i$ for some hyperplane class $h$. Using the Perron-Frobenius theory of non-negative matrices, Conjecture $\mathcal{O}$ reduces to a graph-theoretic check of the quantum Bruhat graph. The techniques involving Perron-Frobenius theory used by Li, Mihalcea, and Shifler in \cite{LMS} and Cheong and Li in \cite{CL} imply the following lemma:

\begin{lemlab}\label{lemma: propO}
If the following conditions hold for a Fano variety $F$:
\begin{enumerate}
\item the matrix representation of $\hat{c}_1$ is nonnegative,
\item the quantum Bruhat graph of $F$ is strongly connected, and 
\item there exists a cycle of length $r$, the Fano index, in the quantum Bruhat graph of $F$,
\end{enumerate}
then Property $\mathcal{O}$ holds for $F$. We say the matrix representation of $\hat{c}_1$ is nonnegative if all of the entries are nonnegative.
\end{lemlab}

We refer the reader to \cite[section 4.3]{Minc} for further details on Perron-Frobenius theory.

\section{Checking Property $\mathcal{O}$ Holds}
Let $X$ be a horospherical variety. We will simplify our notation where the basis of $H^{\bullet}(X)$ is $\{1,h,\alpha_i\}_{i\in I}$ for some finite index set $I$.  Observe by \cite{GPPS} that  the anticanonical classes are \[ c_1(X) = \left\{ \begin{array}{ll}
            5h & \text{ when X is case (1) for }n=3 \\
            7h & \text{ when X is case (2)}\\
            4h & \text{ when X is case (5)}
        \end{array}
    \right.\]
and the Fano indices are
 \[ r = \left\{ \begin{array}{ll}
            5 & \text{ when X is case (1) for }n=3 \\
            7 & \text{ when X is case (2)}\\
            4& \text{ when X is case (5)}
        \end{array}.
    \right.\]   
The endomorphism $\hat{c}_1$ acting on the basis elements of $H^{\bullet}(X)$ is determined by the Chevalley formula in the following way:
\begin{eqnarray*}
            \hat{c}_1(\alpha_i)&=&5(h\star \alpha_i)|_{q=1}  \text{ when X is case (1) for }n=3, \\
            \hat{c}_1(\alpha_i)&=&7(h\star \alpha_i)|_{q=1}   \text{ when X is case (2)}, \text{and}\\
            \hat{c}_1(\alpha_i)&=&4(h\star \alpha_i)|_{q=1}   \text{ when X is case (5)}.
\end{eqnarray*}

Each of the following three subsections will show that Conjecture $\mathcal{O}$ holds for case (1) for $n=3$, case (2), and case (5) of Pasquier's list, respectively. In each subsection we will reformulate the quantum Chevalley formulas stated in \cite{GPPS}, present the quantum Bruhat graph, and argue that each condition of Lemma \ref{lemma: propO} is satisfied. For each case, we have kept the same format of the equations presented by \cite{GPPS} with our prescribed basis for ease of identification for the reader. For example, line 3 in \cite[Proposition 4.3]{GPPS} is \[h*\sigma'_{u_2}=\sigma'_{u_3}+\sigma'_{u'_3} \mbox{ } and \mbox{ } h* \sigma'_{u'_2}=2\sigma'_{u'_3}+\tau_{v_0}. \] In Proposition 1 below we identify this line with \[\hat{c}_1(\alpha_1)=5\alpha_3+5\alpha_4 \mbox{ } and \mbox{ } \hat{c}_1(\alpha_2)=10\alpha_3+5\alpha_5. \]

\subsection{Case (1) for $n=3$}  We will reformulate the quantum Chevalley formula stated in \cite{GPPS} using the basis $\{1,h, \alpha_1, \alpha_2,\cdots, \alpha_{18}\}$.

\begin{proplab}\label{Gon4.3} The following equalities hold by \cite[Proposition 4.3]{GPPS}.
\begin{enumerate}
\item $\hat{c}_1(1)=5h$
\item $\hat{c}_1(h)=10\alpha_1+5\alpha_2$
\item $\hat{c}_1(\alpha_1)=5\alpha_3+5\alpha_4$ and $\hat{c}_1(\alpha_2)=10\alpha_3+5\alpha_5$
\item $\hat{c}_1(\alpha_3) = 10\alpha_6+5\alpha_7+5\alpha_8, \ \ \hat{c}_1(\alpha_4)=5\alpha_6 + 10\alpha_7, $ and $\hat{c}_1(\alpha_5)=5\alpha_8$
\item $\hat{c}_1(\alpha_6)=10\alpha_9+5\alpha_{10}+5\alpha_{11},\ \ \hat{c}_1(\alpha_7) = 5\alpha_{10}$ and $\hat{c}_1(\alpha_8)=5\alpha_{11}+5\cdot1$
\item $\hat{c}_1(\alpha_9)=5\alpha_{12}+5\alpha_{13}, \ \ \hat{c}_1(\alpha_{10})=10\alpha_{13}+5\alpha_{14}\ \ \hat{c}_1(\alpha_{11})=5\alpha_{12}+5\alpha_{14}+5h$
\item $\hat{c}_1(\alpha_{12})=5\alpha_{15}+5\alpha_1, \ \ \hat{c}_1(\alpha_{13})=5\alpha_{15} +5 \alpha_{16},$ and $\hat{c}_1(\alpha_{14}) = 5\alpha_{15}+5\alpha_2$
\item $\hat{c}_1(\alpha_{15})=5\alpha_{17}+5\alpha_3$ and $\hat{c}_1(\alpha_{16}) =5 \alpha_{17}+5\alpha_5$
\item $\hat{c}_1(\alpha_{17})=5 \alpha_{18}+5\alpha_6+5\alpha_8$
\item $\hat{c}_1(\alpha_{18}) =5 \alpha_9+5\alpha_{11}+10\cdot1$
\end{enumerate}
\end{proplab}

\label{ex: case(1)}
The following figure is the quantum Bruhat graph of the Fano variety $X$ in case (1) for $n=3$. Colored edges are introduced in this figure to improve readability. The bold edges indicate a cycle of length $r=5$, the Fano index.
\begin{figure}[H]
\begin{center}
\caption{}
\label{qcbg:1-3}
\scalebox{.75}{
\begin{tikzpicture}
\tikzset{edge/.style = {->,> = latex'}}
    \node (0) at (0,0) {$1$};
    \node (1) at (0,-1) {$h$};
    \node (2) at (-1,-2) {$\alpha_1$};
    \node (3) at (1,-2) {$\alpha_2$};
    \node (4) at (-2,-3) {$\alpha_3$};
    \node (5) at (0,-3) {$\alpha_4$};
    \node (6) at (2,-3) {$\alpha_5$};
    \node (7) at (-2,-4) {$\alpha_6$};
    \node (8) at (0,-4) {$\alpha_7$};
    \node (9) at (2,-4) {$\alpha_8$};
    \node (10) at (-2,-5) {$\alpha_9$};
    \node (11) at (0,-5) {$\alpha_{10}$};
    \node (12) at (2,-5) {$\alpha_{11}$};
    \node (13) at (-2,-6) {$\alpha_{12}$};
    \node (14) at (0,-6) {$\alpha_{13}$};
    \node (15) at (2,-6) {$\alpha_{14}$};
    \node (16) at (-1,-7) {$\alpha_{15}$};
    \node (17) at (1,-7) {$\alpha_{16}$};
    \node (18) at (0,-8) {$\alpha_{17}$};
    \node (19) at (0,-9) {$\alpha_{18}$};
    \draw [edge] (0) to (1);
    \draw [edge] (1) to (2);
    \draw [edge] (1) to (3);
    \draw [edge] (2) to (4);
    \draw [edge] (2) to (5);
    \draw [edge] (3) to (4);
    \draw [edge] (3) to (6);
    \draw [edge] (4) to (7);
    \draw [edge] (4) to (8);
    \draw [edge] (4) to (9);
    \draw [edge] (5) to (7);
    \draw [edge] (5) to (8);
    \draw [edge] (6) to (9);
    \draw [edge] (7) to (10);
    \draw [edge] (7) to (11);
    \draw [edge] (7) to (12);
    \draw [edge] (8) to (11);
    \draw [edge] (9) to (12);
    \draw [blue, edge] (9) to [bend right=100] (0);
    \draw [edge] (10) to (13);
    \draw [edge] (10) to (14);
    \draw [edge] (11) to (14);
    \draw [edge] (11) to (15);
    \draw [edge] (12) to (13);
    \draw [edge, ultra thick] (12) to (15);
    \draw [red, edge] (12) to [bend right=100] (1);
    \draw [edge] (13) to (16);
    \draw [green,edge] (13) to [bend left=100] (2);
    \draw [edge] (14) to (16);
    \draw [edge] (14) to (17);
    \draw [edge, ultra thick] (15) to (16);
    \draw [pink, edge] (15) to [bend right=100] (3);
    \draw [edge, ultra thick] (16) to (18);
    \draw [orange, edge] (16) to [bend left=100] (4);
    \draw [edge] (17) to (18);
    \draw [yellow, edge] (17) to [bend right=100] (6);
    \draw [edge, ultra thick] (18) to (19);
    \draw [purple, edge] (18) to [bend left =100] (7);
    \draw [teal, edge] (18) to [bend right=100] (9);
    \draw [lime, edge] (19) to [bend left=100] (10);
    \draw [edge, ultra thick] (19) to [bend right=100] (12);
    \draw [cyan, edge] (19) to [bend right=100] (0);
\end{tikzpicture}
}
\end{center}
\end{figure}

\begin{lemlab}\label{lemma:case1}
Property $\mathcal{O}$ holds when $X$ is case (1) with $n=3$ of Pasquier's list.
\end{lemlab}
\begin{proof}
The coefficients that appear in the equations in Proposition \ref{Gon4.3} are the entries of the matrix representation of $\hat{c}_1$. Therefore, the matrix representation of $\hat{c}_1$ is nonnegative. The quantum Bruhat graph is strongly connected by Figure \ref{qcbg:1-3}, and the cycle $\alpha_{18}\alpha_{11}\alpha_{14}\alpha_{15}\alpha_{17}\alpha_{18}$ has length $r=5$.
\end{proof}

\subsection{Case (2)} Again, we reformulate the quantum Chevalley formula from \cite{GPPS} using the basis $\{1,h, \alpha_1, \alpha_2,\cdots, \alpha_{12}\}$.
\begin{proplab}\label{Gon4.4} The following equalities hold by \cite[Proposition 4.4]{GPPS}.
\begin{enumerate}
\item $\hat{c}_1(1)=7h$
\item $\hat{c}_1(h)=7\alpha_1$
\item $\hat{c}_1(\alpha_1)=14\alpha_2+7\alpha_3$
\item $\hat{c}_1(\alpha_2) = 7\alpha_4+7\alpha_5$ and $\hat{c}_1(\alpha_3)=7\alpha_5$
\item $\hat{c}_1(\alpha_4)=7\alpha_6+7\alpha_7$ and $\hat{c}_1(\alpha_5)=7\alpha_7$
\item $\hat{c}_1(\alpha_6)=7\alpha_8$ and $\hat{c}_1(\alpha_7)=7\alpha_8+7\alpha_9$
\item $\hat{c}_1(\alpha_8)=7\alpha_{10}$ and $\hat{c}_1(\alpha_9) =7 \alpha_{10}+7\cdot1$
\item $\hat{c}_1(\alpha_{10})=7\alpha_{11}+7h$
\item $\hat{c}_1(\alpha_{11})=7 \alpha_{12}+7\alpha_1$
\item $\hat{c}_1(\alpha_{12}) =7 \alpha_2$
\end{enumerate}
\end{proplab}

The quantum Bruhat graph is 
\begin{figure}[H]
\caption{}
\label{qcbg:2}
\begin{center}
\scalebox{.75}{
\begin{tikzpicture}
\tikzset{edge/.style = {->,> = latex'}}
    \node (0) at (0,0) {$1$};
    \node (1) at (0,-1) {$h$};
    \node (2) at (0,-2) {$a_1$};
    \node (3) at (1,-3) {$a_3$};
    \node (4) at (-1,-3) {$a_2$};
    \node (5) at (1,-4) {$a_5$};
    \node (6) at (-1,-4) {$a_4$};
    \node (7) at (1,-5) {$a_7$};
    \node (8) at (-1,-5) {$a_6$};
    \node (9) at (2,-6) {$a_9$};
    \node (10) at (1,-6) {$a_8$};
    \node (11) at (1,-7) {$a_{10}$};
    \node (12) at (1,-8) {$a_{11}$};
    \node (13) at (1,-9) {$a_{12}$};
    \draw [edge] (0) to (1);
    \draw [edge] (1) to (2);
    \draw [edge] (2) to (3);
    \draw [edge] (2) to (4);
    \draw [edge] (3) to (5);
    \draw [edge] (4) to (5);
    \draw [edge, ultra thick] (4) to (6);
    \draw [edge] (5) to (7);
    \draw [edge] (6) to (7);
    \draw [edge, ultra thick] (6) to (8);
    \draw [edge] (7) to (9);
    \draw [edge] (7) to (10);
    \draw [edge, ultra thick] (8) to (10);
    \draw [edge] (9) to [bend right] (0);
    \draw [edge] (9) to (11);
    \draw [edge, ultra thick] (10) to (11);
    \draw [edge] (11) to [bend right=100] (1);
    \draw [edge, ultra thick] (11) to (12);
    \draw [edge] (12) to [bend right=100] (2);
    \draw [edge, ultra thick] (12) to (13);
    \draw [edge, ultra thick] (13) to [bend left=100] (4);

\end{tikzpicture}
}
\end{center}
\end{figure}

\begin{lemlab} \label{lemma:case2}
Property $\mathcal{O}$ holds when $X$ is case (2) of Pasquier's list. 
\end{lemlab}
\begin{proof}
The coefficients that appear in the equations in Proposition \ref{Gon4.4} are the entries of the matrix representation of $\hat{c}_1$. Therefore, the matrix representation of $\hat{c}_1$ is nonnegative. The quantum Bruhat graph is strongly connected by Figure \ref{qcbg:2}, and the cycle $\alpha_{12}\alpha_2\alpha_4\alpha_6\alpha_8\alpha_{10}\alpha_{11}\alpha_{12}$ has length $r=7$.
\end{proof}

\subsection{Case(5)} Again, we reformulate the quantum Chevalley formula from \cite{GPPS} using the basis $\{1,h, \alpha_1, \alpha_2,\cdots, \alpha_{10}\}$.
\begin{proplab}\label{Gon4.6} The following equalities hold by \cite[Proposition 4.6]{GPPS}.
\begin{enumerate}
\item $\hat{c}_1(1)=4h$
\item $\hat{c}_1(h)=12\alpha_1+4\alpha_2$
\item $\hat{c}_1(\alpha_1)=8\alpha_3+4\alpha_4$ and $\hat{c}_1(\alpha_2)=4\alpha_4$
\item $\hat{c}_1(\alpha_3) = 12\alpha_5+4\alpha_6$ and $\hat{c}_1(\alpha_4)=4\alpha_6+4\cdot1$
\item $\hat{c}_1(\alpha_5)=4\alpha_7+4\alpha_8$ and $\hat{c}_1(\alpha_6)=8\alpha_7+4h$
\item $\hat{c}_1(\alpha_7)=4\alpha_9+4\alpha_1$ and $\hat{c}_1(\alpha_8)=4\alpha_9+4\alpha_2$
\item $\hat{c}_1(\alpha_9)=4\alpha_{10}+4\alpha_3+4\alpha_4$
\item $\hat{c}_1(\alpha_{10})=4\alpha_5+4\alpha_6+8\cdot 1$
\end{enumerate}
\end{proplab}

The associated quantum Bruhat graph is 
\begin{figure}[H]
\caption{}
\label{qcbg:5}
\begin{center}
\scalebox{.75}{
\begin{tikzpicture}
\tikzset{edge/.style = {->,> = latex'}}
    \node (0) at (0,0) {$1$};
    \node (1) at (0,-1) {$h$};
    \node (2) at (-1,-2) {$\alpha_1$};
    \node (3) at (1,-2) {$\alpha_2$};
    \node (4) at (-1,-3) {$\alpha_3$};
    \node (5) at (1,-3) {$\alpha_4$};
    \node (6) at (-1,-4) {$\alpha_5$};
    \node (7) at (1,-4) {$\alpha_6$};
    \node (8) at (-1,-5) {$\alpha_7$};
    \node (9) at (1,-5) {$\alpha_8$};
    \node (10) at (-1,-6) {$\alpha_9$};
    \node (11) at (-1,-7) {$\alpha_{10}$};
    \draw [edge] (0) to (1);
    \draw [edge] (1) to (2);
    \draw [edge] (1) to (3);
    \draw [edge] (2) to (4);
    \draw [edge] (2) to (5);
    \draw [edge] (3) to (5);
    \draw [edge] (4) to (6);
    \draw [edge] (4) to (7);
    \draw [edge] (5) to (7);
    \draw [edge] (5) to [bend right=100] (0);
    \draw [edge] (6) to (8);
    \draw [edge] (6) to (9);
    \draw [edge, ultra thick] (7) to (8);
    \draw [edge] (7) to [bend right=100] (1);
    \draw [edge] (8) to [bend left=100] (2);
    \draw [edge, ultra thick] (8) to (10);
    \draw [edge] (9) to [bend right=100] (3);
    \draw [edge] (9) to (10);
    \draw [edge] (10) to [bend right=100] (5);
    \draw [edge, ultra thick] (10) to (11);
    \draw [edge] (10) to [bend left=100] (4);
    \draw [edge, ultra thick] (11) to [bend right=100] (7);
    \draw [edge] (11) to [bend left=100] (6);
    \draw [edge] (11) to [bend left=100] (0);

\end{tikzpicture}
}
\end{center}
\end{figure}

\begin{lemlab}\label{lemma:case5}
Property $\mathcal{O}$ holds when $X$ is case (5) of Pasquier's list. 
\end{lemlab}
\begin{proof}
The coefficients that appear in the equations in Proposition \ref{Gon4.6} are the entries of the matrix representation of $\hat{c}_1$. Therefore, the matrix representation of $\hat{c}_1$ is nonnegative. The quantum Bruhat graph is strongly connected by Figure \ref{qcbg:5}, and the cycle $\alpha_{10}\alpha_6\alpha_7\alpha_9\alpha_{10}$ has length $r=4$.
\end{proof}
Theorem \ref{thm:main} follows from Lemmas \ref{lemma:case1}, \ref{lemma:case2}, \ref{lemma:case5}, and the previously mentioned work done by Li, Mihalcea, Shifler for the odd symplectic Grassmannian case in \cite{LMS}.


\begin{thebibliography}{1}

\bibitem {CL} D. ~Cheong, C.~ Li, \emph{On the Conjecture $\mathcal{O}$ of GGI for G/P.} Advances in Mathematics, 306 (2017), 704-721.

\bibitem{GGI} S. ~Galkin, V. ~Golyshev, and H. ~Iritani, \emph{Gamma Classes and Quantum Cohomology of Fano Manifolds: Gamma Conjectures.} Duke Mathematical Journal, 165 (2016) no. 11, 2005-2077.

\bibitem{GPPS}
R.~Gonzales, C.~Pech, N.~Perrin and A. Samokhin, \emph{Geometry of Horospherical Varieties of Picard Rank One,} (2018), arXiv:1803.05063.

\bibitem{LMS} C.~Li, L. Mihalcea, and R. Shifler, \emph{Conjecture $\mathcal{O}$ Holds for the Odd Symplectic Grassmannian,} Bull. London Math. Soc. 51 (2019) 705-714. 

\bibitem{Minc} H. ~Minc. \emph{Nonnegative matrices.} (1988), Wiley.


\bibitem{P} B. ~Pasquier, \emph{On Some Smooth Projective Two-orbit Varieties with Picard Number 1.} Mathematische Annalen, 344 (2009) no. 4, 963-987.

\bibitem{Pthesis} B. ~Pasquier, \emph{Vari\'et\'es horosph\`eriques de Fano.} Available at \url{http://tel.archives-ouvertes.fr/docs/00/11/60/77/PDF/Pasquier2006/pdf}.

\bibitem{delpezzo}
J.~Hu, H.~Ke, C. Li, and T. Yang, \emph{Gamma Conjecture I for Del Pezzo Surfaces,} (2010), arXiv:1901.01748.

\bibitem{projinter}
H.~Ke, \emph{On Conjecture $\mathcal{O}$ for Projective Complete Intersections,} (2018), arXiv:1809.10869.

\end{thebibliography}
\end{document}